\documentclass[11pt,a4paper]{article}
\usepackage{amsmath}
\usepackage{geometry}
\usepackage{amsthm}
\usepackage{amsfonts}
\usepackage{amssymb}
\usepackage{latexsym}
\newtheorem{thm}{Theorem}[section]
\newtheorem{cor}[thm]{Corollary}
\newtheorem{prop}[thm]{Proposition}

\newtheorem{Def}[thm]{Definition}

\newcommand{\be}{\begin{equation}}
\newcommand{\ee}{\end{equation}}
\newcommand{\ben}{\begin{enumerate}}
\newcommand{\een}{\end{enumerate}}
\newcommand{\beq}{\begin{eqnarray}}
\newcommand{\eeq}{\end{eqnarray}}
\newcommand{\beqn}{\begin{eqnarray*}}
\newcommand{\eeqn}{\end{eqnarray*}}

\newcommand{\NL}{\mathbb{N}}

\newcommand{\wt}{\widetilde}
\newcommand{\spanned}[1]{\left\langle#1\right\rangle}

\DeclareMathOperator{\cdim}{\mathcal{C}dim}

\let\emptyset\varnothing
\let\le\leqslant
\let\leq\leqslant
\let\ge\geqslant
\let\geq\geqslant

\def\cycle{\mathfrak{C}}
\def\degc{\deg^c}

\def\limfunc{\big\rfloor}
\def\generator{\mathcal}
\newcommand\generatorMap[2][\varphi]{#1_{\vphantom{\sum}_{\!\!\mathcal{#2}}}}
\newcommand{\charac}[2][\varphi]{#1_{\vphantom{\sum}_{\!\!#2}}}
\newcommand{\ts}[1]{_{_{#1}}}
\let\ov\overline

\title{A Simple Proof for the Cycle Double Cover Conjecture}
\author{A. Tayebi and A. Alipour}

\begin{document}

\maketitle

\begin{abstract}
Given a bridgeless graph $G$, the well-known cycle double cover conjecture posits that there is a list
of cycles of $G$, such that every edge appears in exactly two cycles. In this paper, we prove the cycle double cover conjecture. More precisely, we prove the  Goddyn's conjecture as a stronger version of  the cycle double cover conjecture which states that every cycle in $G$ is a  member of some  cycle double cover of $G$.
\end{abstract}

\section{Introduction}
A cycle double cover is a collection of cycles in an undirected graph that together include each edge of the graph exactly twice. The cycle double cover conjecture was introduced independently by Szekeres in  1973 and Seymour in 1979, whether every bridgeless graph has a cycle double cover (see \cite{S1} and \cite{S2}). This beautiful conjecture is now widely considered to be among the most important open problems in graph theory.

The conjecture can equivalently be formulated in terms of graph embeddings, and in that context is also known as the circular embedding conjecture. Indeed, if a graph has a cycle double cover, then the cycles of the cover can be used to form the 2-cells of a graph embedding into a two-dimensional cell complex. In the case of a cubic graph, this complex always forms a manifold. The graph is said to be circularly embedded into the manifold, in that every face of the embedding is a cycle in the graph. However, a cycle double cover of a graph with degree greater than three may not correspond to an embedding on a manifold: the cell complex formed by the cycles of the cover may have non-manifold topology at its vertices. The circular embedding conjecture or strong embedding conjecture states that every biconnected graph has a circular embedding into a manifold \cite{J}. 
Moreover, Goddyn's conjecture asserting if $C$ is a cycle  in bridgeless graph $G$, there is a cycle double cover $G$ which contains  $C$ (see \cite{G1} and \cite{G2}).

In this paper, we will prove the Goddyn's conjecture.
\begin{thm}\label{MainTHM}
The Goddyn's conjecture is true. In particular, the cycle double cover conjecture holds.
\end{thm}

\section{Preliminaries}
A  graph $G$ is a triple $(V, E, r)$,  where $V$ is a finite set of
vertexes, $E$ is a finite set of edges and $r$ is a map from $E$ into non-empty subsets of $V$ with
at most two elements.
A graph $K=(V_K, E_K, r_K)$ is called a subgraph of $G=(V,E,r)$ if $V_K\subseteq V$, $E_K\subseteq E$ and
$r|_{E_K}=r_K$. Furthermore,  if $r(e)\subseteq V_K$ implies $e\in E_K$, then $K$ is called an induced subgraph of $G$. Thus, a subgraph is uniquely determined when its vertices and edges are given. In an induced subgraph, one just needs to determine the set of its vertices to completely define it.

A walk in $G$ is an alternating sequence   $w=\{w_i\}_{i = 0}^{2k}$ such that
\begin{align*}
w_{2i} &\in V,\quad &\forall\ 0 \leq i \leq k,\\
w_{2i+1} \in E\ \ \ &\text{and}\ \ \ r(w_{2i+1}) = \{ w_{2i},w_{2i+2} \}, \quad &\forall\ 0 \leq i < k.
\end{align*}
Here, we say that $w$ is a walk from $w_0$ to  $w_{2k}$. Also, $w_0$ and $w_{2k}$ are called the endpoints of $w$, and
$k$ is length of $w$. Note that $k = 0$  means that there is no edge presented in $w$.
For any vertex $w_{2j}\ (0 \leq j \leq k)$, we can naturally define the \emph{left} and \emph{right sides} of $w$ with respect to $w_{2j}$  as
$\{ w_i\}_{i = 0}^{2j}$ and $\{ w_{2j+i}\}_{i = 0}^{2k-2j}$, respectively.
Obviously, there is a walk from each vertex to itself (with length 0).

Let us remark the  equivalence relation $\thicksim$ on set of vertices of $G$  as follows:
\[
v_0\thicksim v_1 \quad \equiv \quad\text{there exists a walk from $v_0$ to $v_1$}.
\]
Then $\thicksim$ classifies  $G$ into its connected components. The number of connected components of $G$ is denoted by
$\mathfrak{N}(G)$. Also, $\mathfrak{N}_0(G)$ is denoted the number of connected components of $G$ with at least one edge.

A graph is called connected if it has only one connected component or equivalently
for any two distinct vertices of $G$ there is a walk connecting them.
A bridge in $G$ is an edge $e$ of $G$  such that if $e$ is removed from $G$, then the connected components of $G$ will increase.
Also, a graph is called bridgeless if it has not any bridge.

A walk $w=\{w_i\}_{i = 0}^{2k}$ in a graph $G$ is called closed if $w_0 = w_{2k}$,
otherwise it is called an open walk.  The inverse of walk $w$ is defined as a walk with reverse traverse direction of $w$. More precisely, $w^{-1}:=\{ w^{-1}_i\}_{i = 0}^{2k}$ and $w^{-1}_i:=w_{2k-i}\ (0\leq i\leq 2k)$. A trail is a walk without any repeated
edge. Here, a path is a trail without any repeated vertex. Also, a cycle
is a closed trail without any repeated vertex, except for the  first and last one.
\begin{Def}
Let   $G=(V, E, r)$ be a graph.
\item[](i) A  piece is a path $w=\{w_i\}_{i = 0}^{2k}$ such that $k<2$ or  for every
vertex $w_{2j}\ (0 < j < k)$ there exists a cycle in $G$
which has common edges with left and right sides of $w$ with respect to
$w_{2j}$.
\item[](ii) A body of a walk $w$ denoted by
$G(w)$ and defined as the subgraph of $G$ such that $\{w_{2i}\}_{i = 0}^{k}$ and $\{w_{2i+1}\}_{i = 0}^{k-1}$ are its  vertices and edges sets, respectively.
\item[](iii) A piece $w$ is called maximal in path $\wt w$ if $G(w)\subseteq G(\wt w)$ and any piece $\ov w$ with $G(w)\subseteq G(\ov w)\subseteq G(\wt w)$
satisfies  $G(w) = G(\ov w)$;
\item[](iv) The set $\cycle(G)$ is defined by following
\[
\cycle(G):= \big\{\ G(w)\ \big|\ \text{$w$ is a cycle in $G$}\ \big\}.
\]
Also, let $\cycle(v)$ denotes the set of all of cycle bodies in $G$ containing $v$. More precisely,
\[
\cycle(v):=\big\{\ C\in \cycle(G)\ \big|\ %
 \text{$v$ is a vertex of $C$} \ %
\big\}.
\]
\end{Def}

Let $ {w}=\{{w}_i\}_{i = 0}^{2{k}}$ and $\ov {w}=\{\ov{w}_i\}_{i = 0}^{2\ov{k}}$ be two walks such that $w_{2k} = \ov{w}_{0}$. Then one can join $w$ and $\ov {w}$ in the following natural way
\[
w + \ov{w}:=\{\wt{w}_i\}_{i = 0}^{2(k+\ov{k})},
\]
where
\begin{align*}
\wt{w}_i:= \left\{%
\begin{array}{lc}
w_i, \qquad & 0\leq i\leq 2k,\\
\\
\ov{w}_{i-2k},\qquad & 2k < i \leq 2(k+\ov{k}).
\end{array}
\right.
\end{align*}

The union and intersection of two graphs $G=(V,E,r)$ and $G_0=(V_0,E_0,r_0)$ are defined as follows
\[
G\cup G_0:=(V \cup V_0, E \cup E_0, r\cup r_0), \  \  \quad G\cap G_0:=(V \cap V_0, E \cap E_0, r\cap r_0).
\]
Note that, the function $r$ is  a set of pairs. Let empty set be considered to be a map from empty set to whatever set. Then the triple $(\emptyset, \emptyset, \emptyset)$  considered as the null graph.  Also,  $G_0\subseteq G$ means that $V_0\subseteq V, E_0\subseteq E$ and
$r_0\subseteq r$. This is equal to say that $G_0$ is a subgraph of $G$.

Let $G = (V,E,r)$ and $\bar{G} = (\bar{V}, \bar{E}, \bar{r})$ be two arbitrary graphs. Define  $K:=G-\bar{G}$ by  $K=(V\ts{\!K},E\ts{\!K},r\ts{\!K})$  such that  $E\ts{\!K}:=E-\bar{E}$,
$r\ts{\!K}:=r\big|_{\bar{E}}$ and $V\ts{\!K}$ defined as the  composed of those vertices of $G$ which are not present in $\bar{G}$, or has at least one incident edge out of $\bar{G}$. More precisely,
\[
V_k:=(V-\bar{V}) \cup \Big(\bigcup_{e\in \bar{E}}r^{-1}(e) \Big).
\]
Let $G = (V,E,r)$ be a connected graph. Suppose that $H_1 = (V_1, E_1, r_1)$ and $H_2 = (V_2, E_2, r_2)$ are two arbitrary graphs such that $H_2$ is obtained by $H_1$ after removing  all of single vertex connected components of $H_1$.  Then  $G-H_1 = G-H_2$.

Let $G = (V,E,r)$ be a graph. Then $\deg(v)$ denotes the degree of vertex $v$ in $G$.

\begin{Def}\label{sign-labeling}
Let $G = (V,E,r)$ be a graph. Then, $G^c = (V^c,E^c,r^c)$ denotes the union of all cycles in $G$. In this case, we have
\begin{equation}\label{def:G^c}
G^c:=\bigcup_{C\in \cycle(G)} C.
\end{equation}
It is easy to see that,  (\ref{def:G^c})  defines  $V^c$, $E^c$ and $r^c$.

Let us define $\degc(v)$ as follows:
\begin{equation}
\degc(v):=\left\{%
\begin{array}{lcl}
\deg(v) \ \text{in $G^c$,}&\hspace{1.5cm} & \text{ $v\in V^c$,}\\
& &\\
0,&\hspace{1.5cm} &\text{ $v\notin V^c$.}
\end{array}%
\right.
\end{equation}
\end{Def}
\begin{Def}
Let $G = (V,E,r)$ be a graph and  $H=(V_H, E_H, r_H)\subseteq G$. A labeling of $H$ in $G$ is a
map $g:V\longrightarrow\{-1, 1, 0\}$ such that
\[
g(e) = 0 \Longleftrightarrow e\notin E_H.
\]
A function $f:\cycle(G) \times E\rightarrow\{-1,1,0\}$ is called  a
sign labeling if
\[
f(C,e) = 0 \Longleftrightarrow \text{$e$ is not in $C$}.
\]
Note that for every $C\in \cycle(G)$, $f(C)$ is defined as follows
\begin{equation}\label{cycle-to-vector}
\left\{%
\begin{array}{ll}
f(C):&E\longrightarrow \{-1, 1, 0\},\\
&\\
&e\longmapsto f(C,e).
\end{array}
\right.
\end{equation}
\end{Def}

It follows that $f(C)$ is a labeling. Also, it can be seen as a vector in $\mathbb{R}^E$.
With this notation, every sign labeling of a graph will
induce a map from $\cycle(G)$ to $\mathbb{R}^E$ which maps $C\mapsto f(C)$. This map was also
denoted by $f$.

\begin{Def}\label{def:cycle-generic}
Let $G = (V,E,r)$ be a graph. Define the relation $\le_*$ on $V$ as follows
\[
v_1\leqslant_* v_2 \Longleftrightarrow \cycle(v_1)\subseteq \cycle(v_2).
\]
A vertex $v$ in $G$ is  called a cycle generic vertex if  $v$ be maximal with respect to $\le_*$ or $\degc(v) \ge 3$.
\end{Def}
\begin{Def}\label{circular-dimension}
Let $G=(V,E,r)$ be a graph. The cyclic dimension of $G$ will be denoted by $\cdim (G)$ and defined as follows
\begin{equation}\label{c-dim-def-equation}
\cdim (G):=\min \Big\{ \dim \spanned{f\big( \cycle(G) \big)}\ \Big|\ %
\text{$f$ is a sign labeling of $G$} \Big\},
\end{equation}
where $\spanned{A}$ denotes the span of set $A$.
\end{Def}

Note that by the definition of a sign labeling,
$f\big( \cycle(G) \big)$ is a subset of $\mathbb{R}^E$ and thus (\ref{c-dim-def-equation}) is meaningful. Let for a sign labeling $f$ of $G$ the following holds
\begin{equation}\label{optimal-sign-dist}
\cdim (G)= dim \spanned{f\big( \cycle(G) \big)}.
\end{equation}
Then $f$ is called an optimal sign labeling of $G$. Moreover, let $f\big(\cycle(G)\big) \subseteq\spanned{f(C_i)}_{i=1}^m$ holds,
where $\{C_i\}_{i=1}^m\subseteq  \cycle(G)$. Then $\generator{A}:=\{C_i\}_{i=1}^m$ is called an $f$-generator for $G$ if $m =  \cdim (G)$.

For a non-empty subset $\generator{A}$ of $\cycle(G)$, let us define
$\generatorMap{A}$ as follows
\begin{equation}
\left\{%
\begin{array}{rl}
\generatorMap{A}:&E\longrightarrow\NL\cup\{0\}\\
&e\ \longmapsto\#\left\{C\in\generator{A}\ \big|\ \text{$C$ contains $e$}\ \right\},
\end{array}%
\right.
\end{equation}
where  $\#A$ denotes  the cardinality of set $A$. Then $\generatorMap{A}$ is called the characteristic map of $\generator{A}$.
Also, one can define the following partial order on $E^{\NL_0}$
\begin{equation}
f \le^* g\quad \Longleftrightarrow \quad \big(f(e)>2\Rightarrow f(e)\le g(e),\quad (\forall e \in E)\ \big).
\end{equation}

\begin{Def}\label{path-cycle-segment}
Let $G=(V,E,r)$ be a graph.  A path segment is a walk $w=\{w_i\}_{i = 0}^{2k}$ in
$G$ which  satisfies the following conditions:
\begin{enumerate}
\item There exists a cycle $\ov{w}=\{\ov{w}_i\}_{i = 0}^{2\ov{k}}$ in $G$ such that $G(w)\subseteq G(\ov{w})$;
\item $k\geq 1$. This means that $w$ contains at least one edge;
\item $w_0$ and $w_{2k}$ are cycle generic
vertices. Furthermore, none of $\{w_{2i}\}_{i=1}^{k-1}$ are cycle generic vertices.
\end{enumerate}
\end{Def}
\begin{Def}
Let $G=(V,E,r)$ be a graph and   $H$ be  a subgraph of $G$. Define
\[
\mathcal{S}:=\{K\subseteq G|\ \text{ any cycle with common edges
with $K$, contains $K$}\}.
\]
Then, $H$ is called a cycle segment if it is a maximal member in $\mathcal{S}$ by inclusion relation.
\end{Def}


\begin{Def}
Let $G=(V,E,r)$ be a graph. Let us define  $G_c:=\big(V_c, E_c, r_c\big)$, where $V_c$ is the set of all cycle generic vertices of $G$,
$E_c$ is the set of all $G(w)$ for all path segments $w$ in $G$ and $r_c$ is defined by following
\[
r_c :=\Big\{\big(G(w),\{w_0,w_{2k}\}\big)  \ \big|\
\text{$w=\{w_i\}_{i = 0}^{2k}$ is a path segment in $G$} \ \Big\}.
\]
It is easy to see that $G_c$ is a well-defined graph.
\end{Def}
A cycle double cover for a graph $G$ is an indexed family (not necessarily distinct) of cycle bodies in $G$, namely $\mathcal{B}$,  such that every edge in $G$ appears in exactly two indexed members (possibly identical) of $\mathcal{B}$.
If $\mathcal{B}$ is a cycle double cover, then one can  say  that $\mathcal{B}$ is a cover for the sake of simplicity.

\begin{Def}
Let $G=(V,E,r)$ be a  connected bridgeless graph. Then $G$ is called cactus-free if $\cdim (G)=1$ (i.e., $G$ contains only one cycle)
or each cycle of $G$ has more than one path segment. Also, a cycle in  $G$ which has only one path
segment will be called a leaf cycle.
\end{Def}

The importance of a leaf cycle is that it should  appear  twice in every cycle double cover.

Let
$f:X\longrightarrow\mathbb{R}$ be a real value function on an arbitrary set $X$  and $Y\subseteq X$. Then $f\limfunc_Y:X\longrightarrow  \mathbb{R}$ is defined as  follows:
\[
f\limfunc_Y(x):=\left\{%
\begin{array}{lcl}
f(x),& &x\in Y,\\
& &\\
0, & & x\notin Y.
\end{array}
\right.
\]
Now, we are ready to state our results.
\section{Proof of Theorem \ref{MainTHM}}
\begin{prop}\label{prop:first and basic}
Let $G=(V,E,r)$ be a graph and $f$ be an optimal sign labeling of $G$. Then the following hold:
\begin{description}
  \item[] (i) For any path segment $w$ of $G$, the following holds
\begin{align*}
\cdim \big(G - G(w)\big) < \cdim (G).
\end{align*}
  \item[] (ii) Let $\wt w$ be a path in $G$ and $w_1$ and $w_2$ be two maximal pieces of $\wt w$ with at least
one common edge. Then, $G(w_1) = G(w_2)$. In particular, the maximal pieces of $\wt w$
are edge disjoint. Therefore, $\wt w$ can be uniquely  (up to piece bodies)
written as the union of its edges  disjoint maximal pieces.
  \item[] (iii) Let $A\neq \emptyset$ be the set of all common vertices in
cycles $\{C_i\}_{i=1}^s$. Then, $A$ contains at least one cycle generic vertex.
In particular, every cycle has a cycle generic vertex.
  \item[] (iv) Let $P_1\neq P_2$ be two path segments. Then,
there does  not exist any common vertex in both of them in the middle. This means that if $v$
be a vertex common in $P_1$ and $P_2$, then $v$ is a cycle generic vertex. In particular,
every two distinct path segments do not have any common edge. Consequently, if a path segment has some
common edge with a cycle, then it is contained in that cycle.
  \item[] (v)  $\cdim (G_c)=\cdim (G)$. Consequently, $\quad\cdim (G_c)=\cdim (G^c)=\cdim (G)$.
  \item[] (vi) Let $C_1$ and $C_2$ be two members of $\cycle(G)$. Suppose that $H=(V_H,E_H,r_H)\subseteq C_1\cap C_2$ such that $H$ be a connected graph. Then the following holds:
\[
f(C_1)\limfunc_{E_H}\ || \ f(C_2)\limfunc_{E_H},
\]
where $u_1||u_2$ means  two vectors $u_1$ and $u_2$
are parallel, i.e.,  \mbox{$\spanned{u_1}=\spanned{u_2}$}.
  \item[] (vii) Let $w$ be a path in $G$ and $P=G(w)$. Then, there  exists a labeling $g$ for $P$
such that for every path $w'$ in $G$ with same endpoints as $w$, one can find a
labeling $g'$ for $P'=G(w')$  such that
\begin{align}
g(P) + g'(P') = \sum_{i = 1}^{r} \lambda_i f(C_i)\label{P1}
\end{align}
holds for some $\{C_i\}_{i=1}^r$ in $\cycle(G)$. Also,  if $w$ be a piece and $g''$ be another labeling for $w$ with the same property as $g$, then
$g||g''$.
  \item[] (viii) Let $H=(V_H,E_H,r_H)$  be a subgraph of $G$. Then, $f|_{\cycle(H)\times E_H}$
is a optimal sign labeling for $H$.
  \item[] \label{part-one-plus-dimension} (ix) Let $w$ be a path segment of $G$ and define $\widehat{G}:= G-G(w)$. Then the following holds:
\[
\cdim (G)=1+\cdim (\widehat{G}).
\]
\end{description}
\end{prop}
\begin{proof}
We are going to prove the proposition, part by part as follows:\\\\
{\bf Proof of Part (i):} On the contrary,  assume that $\cdim (G-H) = \cdim (G)$. Let $H:=G(w)$. Then, we have
\begin{align*}
\cdim (G)=\cdim (G-H) \leq \dim \spanned{ f\big( \cycle(G-H) \big) } \leq \dim \spanned{ f\big( \cycle(G) \big) } =\cdim (G).
\end{align*}
Then we get $\dim\spanned{ f\big( \cycle(G-H) \big) } =  \dim \spanned{ f\big( \cycle(G) \big) }$. Since $\mathbb{R}^E$ is finite dimension, then  $\spanned{ f\big( \cycle(G-H) \big) } =\spanned{ f\big( \cycle(G) \big) }$. Let $e_0\in E_H$ and $C\in \cycle(G-H)$. Then we get
 \[
 f(C)(e_0) = f(C,e_0) = 0.
 \]
 Thus, for every element $g$ of $\spanned{ f\big( \cycle(G-H) \big) }$æ we have $g(e_0) = 0$.
By definition of path segment, there exists  $C_0\in \spanned{ \cycle(G) }$ which
 contains $H$ (and therefore $e_0$). One can obtain $ f(C_0)(e_0) \neq 0$.
This contradicts with $\spanned{ f\big( \cycle(G-H) \big) } =\spanned{ f\big( \cycle(G) \big) }$. Then, we get the result.\\\\
{\bf Proof of Part (ii):} It is easy to see that  if $G(w_1) \neq G(w_2)$, then one can construct piece $w_3$ such that
$G(w_3) = G(w_1) \cup G(w_2)$. It is remarkable that  every edge out
of $G^c$ is a maximal piece in any path which contains it. Thus,  any path can be uniquely written as the union of its edges  disjoint maximal pieces.\\\\
{\bf Proof of Part (iii):} Let $v$ be an arbitrary vertex of $A$. If $v$ be a cycle generic vertex, then we get the proof. Suppose that  $v$ is not  a cycle generic vertex. Then there exists a cycle generic vertex $v_0$ in $G$ such that
$v\le_* v_0$. This means that every cycle which contains $v$, also contains $v_0$. This implies that
$v_0$ is contained in all of $\{C_i\}_{i=1}^s$ and then the proof is completed.
\\\\
{\bf Proof of Part (iv):}  Let $v$ be a vertex of $G$. If there is not any  cycle which crosses $v$, then  $v$ will
not be contained in any path segment. Let us remark that $\degc(v) \ge 3$ implies that $v$ is a cycle generic. Then, by the definition of path segments, it follows that  $P_1$ and $P_2$ can not have any common vertex in the middle. It is crystal clear that every two distinct path segments have not any common edge. Consequently, if a path segment has some common edge with a cycle, then it is contained in that cycle.

Let a path segment $P_3$ has a common edge $e_0$ with a cycle $C$. Then there exists a path segment $P_4\subseteq C$ such that $e_0\in P_4$. By the same argument used in the latter case, $P_3=P_4$. This completes the proof.
\\\\
{\bf Proof of Part (v):} If there is not any cycle in $G$, then $G$ has not any path segment. Thus,  we have
\[
\cdim (G_c)= 0 =\cdim (G).
\]
Now, assume that there exists  at least one cycle in $G$. In this case,  first we show that the following map is a
bijection
\begin{equation*}
\left\{%
\begin{array}{rl}
\eta:&\cycle(G)\longrightarrow\cycle(G_c),\\
&\\
&C=\Big(\{e_i\} _{i=1}^n,\{v_i\} _{i=1}^{n}\Big)%
\longmapsto \Big(\{P_j\} _{j=1}^r,\{v_{i_j}\} _{j=1}^{r}\Big),
\end{array}
\right.
\end{equation*}
where $\{v_{i_j}\} _{j=1}^{r}$ are cycle generic vertices in
$\{v_i\} _{i=1}^{n}$, and $\{P_j\} _{j=1}^r$ are path segments
between $\{v_{i_j},v_{i_{(j+1)}}\}$ when $(1\le j<r)$ and between $\{v_{i_r},v_{i_1}\}$
when $j=r$.   By part (iii), there exists  at least one cycle generic
vertex in $\{v_i\} _{i=1}^{n}$. Thus, $\eta$ is well defined, injective and surjective.

Let $\pi$ be a projection map from edges of $G^c$ onto path segments. More precisely, for any edge $e$ of $G^c$, $\pi(e)$ is a path segment containing
$e$. By part (iv),  $\pi$ is well defined.

Now, let $f$ be an arbitrary sign labeling of $G$. Suppose that $g$ is a right
inverse  of $\pi$ ($\pi$ is onto and thus $g$ exists). Let $\wt f$ be defined by the following
\begin{equation}
\left\{
\begin{array}{rl}
\wt f:&\cycle(G_c)\times E_c\longrightarrow \{-1,1,0\},\\
&\\
&(C,P)\longmapsto f\Big(\eta^{-1}(C),\ g(P)\Big).
\end{array}
\right.
\end{equation}
In order to prove  that $\wt f$ is a well defined function, one needs to consider that  $\pi\circ g = I_{E_c}$ (which means that
$g(P)$ is an edge of $P$).   We  show that $\wt f$ is a sign
labeling. Equivalently, we show that if a path segment $P$ is not
an edge of $\eta(C)$, then  $\wt f\big(\eta(C),P\big) = 0$ holds. Assume that $P$ is not an edge of $\eta(C)$. Thus
by part (iv), $P$ has not any common edge with $C$. Since $g(P)$ is an edge in $P$, then it is not in $C$. Therefore, we get
\[
\wt f\big(\eta(C),P\big) = f\big(C,g(P)\big) = 0.
\]
The converse is trivial.

Now, we are going to show that if $\sum_{i=1}^{n}\lambda_i f(C_i)=0$, then  the following holds
\begin{equation}\label{equ:relation G implie G_c}
\sum_{i=1}^{n}\lambda_i \wt f\big(\eta(C_i)\big)=0.
\end{equation}
If $\sum_{i=1}^{n}\lambda_i f(C_i)=0$, then we have $\sum_{i=1}^{n}\lambda_i f(C_i,e)=0$, $(\forall e\in E)$. It results that $\sum_{i=1}^{n}\lambda_i f(C_i,g(P))=0$, $(\forall P\in E_c)$. Thus,
\[
\sum_{i=1}^{n}\lambda_i \wt f\big(\eta(C_i),P\big)=0,  \ \ \ \ (\forall P\in E_c)
\]
which yields $\sum_{i=1}^{n}\lambda_i \wt f\big(\eta(C_i)\big)=0$. By (\ref{equ:relation G implie G_c}), it follows that $\cdim (G_c)\leq\cdim (G)$.

Let $\wt f$ be a sign labeling of $G_c$ and let us define $f$ as follows:
\begin{equation*}
f(C,e):=\wt f\big(\eta(C),\pi(e)\big).
\end{equation*}
In order to prove $\cdim (G)\le \cdim (G_c)$, it is sufficient to show that  $\sum_{i=1}^{n}\lambda_i \wt f\big(\eta(C_i)\big)=0$ which implies the following
\begin{equation}\label{equ:relation G_c implie G}
\sum_{i=1}^{n}\lambda_i f(C_i)=0
\end{equation}
It is obvious that $\sum_{i=1}^{n}\lambda_i f(C_i)=0$ holds if and only if $\sum_{i=1}^{n}\lambda_i f(C_i,e)=0\quad (\forall e\in E)$. But it is equivalent to \begin{equation}
\sum_{i=1}^{n}\lambda_i \wt f\big(\eta(C_i),\pi(e)\big)=0, \ \ \ \ \ (\forall e\in E).\label{equ:f}
\end{equation}
The relation (\ref{equ:relation G_c implie G}) can be obtained by  (\ref{equ:f}).
\\\\
{\bf Proof of Parts (vi), (vii) and (viii):}  We prove these parts simultaneously in the following steps:
\begin{enumerate}
\item\label{Step1} In the first step, we prove that if part (viii) holds for $G$, then the  part (vi) holds, also.
\item\label{Step2} In the second step, we show that if part (vi) holds for $G$, then the  part (vii) holds, also.
\item\label{Step3} In the third step, we prove that if the parts (viii) and (vii) hold for graphs $H$ satisfying
\mbox{$\cdim (H) < \cdim (G)$}, then part (viii) holds for $G$.
\item\label{Step4} In the final step, we show that the parts (vi), (vii) and
(viii) hold.
\end{enumerate}
{\bf Step 1}: Put $H:= C_1\cap C_2$ and let $\mathfrak{N}(H)=m$. If $m = 0$, then  $H$ is a
empty graph and we get the proof. In the case of $m = 1$, let us put
\begin{align*}
C_{11}:=C_1-(C_1\cap C_2)\  \ \text{and}\ \  C_{21}:=C_2-(C_1\cap C_2).
\end{align*}
Let us consider $K=(V_K,E_K,r_K):=C_1\cup C_2$. It is easy to see that $\cdim (K) = 2$ and $\cycle(K)=\{C_1,\ C_2,\ C_{11}\cup C_{21} \} $.
Since $f\big|_{ E_K\times\cycle(K)  } $ is an optimal sign labeling for $K$, then we get
\begin{align*}
f(C_1) - f(C_2)= \pm f(C_{11}\cup C_{21}).
\end{align*}
This implies that $f(C_1)\limfunc_{E_H} || f(C_2)\limfunc_{E_H}$, where
$H=(V_H,E_H,r_H)$.

Now, suppose that $(viii)\Rightarrow (vi)$  holds for $m < n$. We are going to prove $(viii)\Rightarrow (vi)$ for $m = n$.
We can assume that $2 \leq n $. Let $A_1$ and $A_2$ be
two distinct connected components of $H$. It is remarkable that  both of
them do not necessarily contain edges. But one can assume that $A_1$ has at least one edge.
We get
\begin{align*}
C_1 = A_1\cup C_{11} \cup A_2 \cup C_{12},\\
C_2 = A_1\cup C_{21} \cup A_2 \cup C_{22},
\end{align*}
where $C_{11}$ and $C_{12}$ are connected components of $C_1 - (A_1\cup A_2)$,
and $C_{21}$ and $C_{22}$ are connected components of $C_2 - (A_1\cup A_2)$. It is easy to see that,  each
$C_{ij}\ (1\leq i,\  j\leq 2)$ can be considered a path body. Without loss of
generality, one can assume that the common endpoints of $C_{11}$ and $A_1$
are different from common endpoints of $C_{22}$ and $A_1$. Note that
each $C_{ij}\ (1\leq i,\ j\leq 2)$ has exactly one common vertex with
$A_k\ (1\leq k\leq 2)$. In fact, this vertex is a common endpoint. We have two main cases as follows:
\item[]Case (1):  The common endpoints of $C_{11}$ and $A_2$ are not equal with the common
endpoints of $C_{22}$ and $A_2$;\\
\item[]Case (2):  The common endpoints of $C_{11}$ and $A_2$ are the same with the common
endpoints of $C_{22}$ and $A_2$.\\

Let us define   $C_3$  as follows
\begin{equation}
C_3:=\left\{
\begin{array}{ll}
A_1\cup C_{11} \cup A_2\cup C_{22},  &\text{in case (a)},\\
&\\
A_1\cup C_{11} \cup C_{22},  &\text{in case (b)}.
\end{array}
\right.
\end{equation}
It is easy to see that, $C_3\in \cycle(G)$. Since the connected components of
$C_1\cap C_3$ and $C_2\cap C_3$ are less than the
connected components of $C_1\cap C_2$, then the connected components of
$C_1\cap C_3$ and $C_2\cap C_3$  are less than $n$. Put
\begin{align*}
E_1:=\text{Edges of}\ (A_1\cup C_{11}\cup A_2), \   E_2:=\text{Edges of}\ (A_1\cup C_{22}\cup A_2) \ \text{and}\ E_0:=\text{Edges of}\ A_1.
\end{align*}
 Then, $f(C_1)\limfunc_{E_1} || f(C_3)\limfunc_{E_1}$ implies $f(C_1)\limfunc_{E_0} || f(C_3)\limfunc_{E_0}$ and $f(C_2)\limfunc_{E_2} || f(C_3)\limfunc_{E_2}$ yields $f(C_2)\limfunc_{E_0} || f(C_3)\limfunc_{E_0}$.  Thus  we get $f(C_1)\limfunc_{E_0} || f(C_2)\limfunc_{E_0}$. Since  $A_1$ is an arbitrary connected component of $H$ containing at least one edge, then we get the proof.\\\\
{\bf Step 2}: Let the part (vi) holds for $G$ and $w=\{w_i\}_{i = 0}^{2k}$ be a piece in $G$.
If there exists a cycle containing $w$, then the part (vi) implies that our
claim holds. Also, if $k= 1$ then we get the proof. Obviously, every piece with length 2 is contained in some cycle. Thus, in the case of  $k = 2$,
$w$ is  contained in some cycle and  part (vi) implies our claim.


Now, let $k > 2$. Suppose that $\ov{w}$ and $\wt{w}$ are obtained by $w$ after removing the first and the last edges, respectively. Note that
if $w$ be contained in some cycle, then we get the proof. So, we assume that
$w$ is not contained in some cycle. Thus, our claim for $\ov{w}$ and $\wt{w}$ holds (as $\ov{w}$ and $\wt{w}$ are pieces).
This means that there exists the  labelings  $\ov{g}$ and $\wt{g}$ for $\ov{w}$ and $\wt{w}$, respectively, as we  desired. But they have at least one common edge
which enables us to construct $g$ on entire $w$ with respect to $\ov{g}$ and $\wt{g}$. Then, if $w$ is a piece, we get the proof.

Now, suppose that $w$ is not a piece. It can be uniquely written as the union of
edge distinct pieces. It is easy to see that the mentioned  union of labeling  $g$ on each
piece has desired properties.\\\\
{\bf Step 3}: If $\cdim (G) < 3$,  then we get the proof. Suppose that the proof holds for the case $\cdim (G) < n+1$. We are going to show that the proof  is also valid for $\cdim (G)= n+1$.
Let $P=G(w)$ be a path segment body. Define $K=(V_K, E_K, r_K):=G-P$. If $w$ be also
a cycle segment, then we obtain
\begin{align*}
\cdim (G) = \cdim (K) + 1.
\end{align*}
Otherwise, there  exists  $C_0\in \cycle(G)$  such that the following holds
\begin{align*}
P\subseteq C_0\ \ \text{and} \ \ \ P_0:=C_0-P.
\end{align*}
Thus
\[
 C_0=P_0\cup P.
\]
Let us define $\wt f$ as follows
\begin{align*}
\wt f(C,e):=%
\left\{\begin{array}{lcl}
f(C, e),& &C\subseteq K,\quad e\in E_K,\\
0,& &C\subseteq K,\quad e\notin E_K,\\
g_0(e)+g(e),& &C=C_0,\\
g'(e)-g(e),& &C\nsubseteq K,\quad C\neq C_0,\\
0,& &C\nsubseteq K,\quad C\neq C_0,\quad e\notin E_K,
\end{array}
\right.
\end{align*}
where $g$ is an arbitrary (but fixed) labeling of $P$,
$g_0$ is a labeling obtained by applying the part (vii) on $P_0$ and $P$ ($C=P_0'\cup P$), and  by (\ref{P1}) we get
\begin{align}
g_0(P_0) + g'(P_0') = \sum_{i = 1}^{r} \lambda_i f(C_i)\label{P2}
\end{align}
which  also defines $g'$.

Now, we  are going to show that  $\dim \spanned{\wt f\big( \cycle(G) \big)} = n+1$ holds. In order to prove it, first note that by (\ref{P2}), we have
\begin{align*}
\wt f(C) + \wt f(C_0) = \sum_{i = 1}^{r} \lambda_i f(C_i)= \sum_{i = 1}^{r} \lambda_i \wt f(C_i)\in \spanned{\wt f\big( \cycle(G) \big)}.
\end{align*}
Also,  $\wt f(C_0)$ is not a member of $\spanned{\wt f\big( \cycle(K) \big)}$.
This means that $\cdim (G) \leq n+1$ and because of $n = \cdim (K) < \cdim (G)$, we get $\cdim (G) = n+1 =\cdim (K)+1$. Then,  $\wt f$ is an optimal sign labeling for $G$.

Also, let $\wt f$ be an arbitrary optimal sign labeling for $G$. By the part (i), we obtain
\begin{align*}
n = \cdim (K)\leq \dim \spanned{\wt f\big( \cycle(K) \big)} \leq n
\end{align*}
which means that $\wt f\big|_{\cycle(K)\times E_K}$ is an optimal sign labeling
for $K$. Up to now,  we prove that the following holds:
\begin{align}
\cdim (G)= \cdim (K)+1.\label{P3}
\end{align}
More precisely, we  get the following:
\begin{center}
\fbox{
  \parbox{.975\textwidth}{
  For each path segment $w$ in $G$, if $K:=G-G(w)$ then the following holds
  \begin{align}
1+\cdim (K)= \cdim (G). \ \ \ \label{P4}
\end{align}
Also,
$\wt f\big|_{\cycle(K)\times E_{K}}$ is an optimal sign labeling  for $K$.
}}
\end{center}

 Let $H=(V_H, E_H ,r_H)$ be an arbitrary subgraph of $G$. If $H = G$, then our
claim holds. So,  assume that $H\subsetneq G$. Obviously, $H^c\neq G$. Then, there exists a path segment $w$ such that $P:=G(w)\cap H^c = \emptyset$
which yields $H^c\subseteq G - P$. Let $\wt f$ be an optimal sign labeling for $G$. By (\ref{P4}), $\wt f_0:=\wt f\big|_{\cycle(G-P)\times E_{G-P}}$ is an optimal sign labeling for $G-P$. By the induction
hypothesis,  $\wt f_0\big|_{\cycle(H^c)\times E_{H^c}}$ is an optimal sign labeling
for $H^c$.  On the other hand, we have $\wt f_0\big|_{\cycle(H^c)\times E_{H^c}}=\wt f\big|_{\cycle(H^c)\times E_{H^c}}$. Therefore, $\wt f\big|_{\cycle(H^c)\times E_{H^c}}$ is an optimal sign labeling
for $H^c$ which implies that $\wt f\big|_{\cycle(H)\times E_{H}}$ is an optimal sign labeling for $H$. \\\\
{\bf Step 4}: In the case of $\cdim (G)<2$, the parts (vi), (vii) and  (viii) hold. By using the inductive method, suppose that if  $\cdim (G)<n$ then (vi), (vii) and  (viii) hold. By the step 3, (viii) holds for $G$ when $\cdim (G)=n$. By steps 1 and 2, it follows that (vi) and (vii) hold for $G$ when $\cdim (G)=n$. This completes the proof.\\\\
{\bf Proof of Part (ix):} By (\ref{P4}),  we get the proof.
\end{proof}

\bigskip
\begin{Def}
A  graph $G=(V, H, r)$ is called \emph{cycle separable} if there exists at least two subgraphs
$G_1=(V_1, E_1, r_1)$ and $G_2=(V_2, E_2, r_2)$ of $G$  such that $G=G_1\cup G_2$, $E_1\cap E_2=\emptyset$, $\#E_1>0$, $\#E_2>0$ and
every $C \in \cycle(G)$ lies in one of $G_1$  or $G_2$. Also, $G$ is called \emph{strong cyclic} if it is not cycle separable.
\end{Def}

Let us define the following relation on  the edges of a particular graph $G = (V,E,r)$
\begin{equation*}
e_1\simeq e_2\qquad \equiv\qquad \text{there exists $C\in \cycle(G)$ such that contains $e_1$ and $e_2$. }
\end{equation*}
This is an equivalence relation on any bridgeless graph.
Let $E_0 \subseteq E$ be an equivalence class of $\simeq$ and
$V_0= \bigcup_{e\in E_0} r(e)$. Then the induced subgraph of $G$ with
vertices $V_0$ and edges $E_0$ is called  a \emph{cycle component} of $G$.

\bigskip

It is easy to see that, the following hold
\begin{enumerate}
\item[] (i) Every strong cyclic graph is cactus free;
\item[] (ii) $G$ is strong cyclic if  it has not any cycle component except $G$;
\item[] (iii) Any bridgeless graph is a union of its (edge disjoint) cycle components;
\item[] (iv) If Goddyn's conjecture holds for the strong cyclic graphs, then  it holds permanently.
\end{enumerate}

\begin{prop}\label{prop:important}
Let $G=(V, E, r)$ be a connected bridgeless graph with $\cdim (G)\ge 2$. Then the following hold:
\begin{enumerate}
\item\label{part:cycle segments and path segments are seperable}%
Let $P$ and $H$ be path and cycle segments in $G$, respectively,   which
are edge disjoints.  Then there exists  $C\in \cycle(G)$ such that
$P\subseteq C$ and $H\nsubseteq C$.
\item\label{part:Graph mines a cycle segment is bridgeless}%
For any cycle segment $H$ in $G$, the graph $\widehat{G}:=G-H$ remains bridgeless.
\item\label{part:each connected component at most one connected component }%
Let $H$ be a cycle segment of $G$ which is not connected. In this case,
for any $C\in \cycle(G)$, each connected component of $G - H$ contains
at most one connected component of $C-H$.
\item\label{part:same number of connected components}%
Let $G$ be a strong cyclic graph. Then,
for any cycle segment $H$ of $G$, the following holds
\begin{equation}
\mathfrak{N}_0(H)=\mathfrak{N}(G-H).\label{P5}
\end{equation}
Moreover, every connected component of
$C-H$ lies in exactly one connected component of $G-H$.
\item\label{part:at most one not connected cycle segment}%
Let $G$ be a strong cyclic graph. Then,
there exists at most one disconnected cycle  segment.
\item\label{part:excistanse of removable path segment}%
Let $G$ be a strong cyclic graph. Then,  there exists a path $P$ such that $G-P$
is a strong cyclic graph. Consequently, if $G$ is a connected
bridgeless graph, then there exists a path $P$ such that $G-P$
is a  connected bridgeless graph.
\item\label{part:existanse of subgraph per dimension}%
Let  $G$ be a strong cyclic graph. Then, for  any $1\le m\le \cdim (G)$,  there exists a  strong cyclic graph $G_m$
such that $G_m\subseteq G$ and $\cdim (G_m)= m $.
\item\label{part:separating cycle-path segment}%
Let $\cdim (G)\ge 3$. Then, for every cycle $C$ in $G$,  there exists a path
segment $P$ such that $C\subseteq G-P$ and $G-P$ is a
connected bridgeless graph.
\item\label{part:path segment removal involve at most 4}%
For any path segment $w$ of $G$, there
exists  at most 4 cycle segments $H$ in $G$ such that $G(w)\subseteq G-H$ and $w$ is not a path segment of
$G-H$.
\item\label{part:path segment-intersection of two cycle}%
Let $G$ be a cactus-free graph. Then,
for any cycle $C$ and any cycle segment $H$ of $G$ which lies in $C$, there exists a
cycle $C_0$ such that $H=C\cap C_0$.
\item\label{part:special generator}
Let $G$ be a strong cyclic graph and $\cdim (G)=n$. Then, for every cycle $C\in \cycle (G)$
and optimal sign labeling $f$,  there exists an $f$-generator
$\generator{A}=\{C_i\}_{i=1}^n$  such that
\begin{enumerate}
\item\label{part:part:a}%
$C_1=C$;
\item\label{part:part:b}%
$\charac{A} \le 2$, where $\charac{A}$ is the characteristic map of $\generator{A}$;
\item\label{part:part:c}%
There exists a map $s:\generator{A}\rightarrow\{-1,1\}$ and cycles $\{C_i'\}_{i=1}^n$ such
that
\begin{equation}
\sum_{i=1}^{t} s(C_i)f(C_i)=f(C_t'),\qquad (\forall \ 1\le t\le n).
\end{equation}
\end{enumerate}
\end{enumerate}
\end{prop}
\begin{proof}
We are going to prove the proposition part by part as follows:\\\\
{\bf Proof of Part \ref{part:cycle segments and path segments are seperable}:}
Since  $P\nsubseteq H$,  then there exists
$C_1\in \cycle(G)$ such that $H\subseteq C_1$ and $P\nsubseteq C_1$. Let $C_2\in \cycle(G)$ such that $P\subseteq C_2$.
If $H\nsubseteq C_2$, then we get the proof. Otherwise, since the degree of every vertex of $K$ is even then  $K:=(C_1\cup C_2) - (C_1\cap C_2)$ is a bridgeless graph.  Obviously,
$P\subseteq K$ and this means that  there exists $C\in \cycle(K)\subseteq \cycle(G)$  such that
$P\subseteq C$. Since $H\nsubseteq C$, then we get the proof.
\\[.3cm]%
{\bf Proof of Part \ref{part:Graph mines a cycle segment is bridgeless}:} Let $H$ be a cycle segment of $G$. One can write
\begin{equation}
\bigcup_{i=1}^s G_i=G-H,
\end{equation}
where $\{G_i\}_{i=1}^s$ are connected components of $G-H$.
We are going to prove that every $G_i\ (1\le i\le s)$ is bridgeless.
On contrary, assume that there exists $1\le i_0\le s$ such that the edge $e$ be a bridge of $G_{i_0}$.
First, we prove that every cycle that contains $e$, also contains $H$.
Let $C$ be a cycle in $G$ which contains $e$. If $C$ does not contain $H$, then $C\subseteq G-H=\cup_{i=1}^s G_i$.
Since $C$ is connected and contains $e$, it follows that $C\subseteq G_{i_0}$. But this means that
$e$ can not be a bridge for $G_{i_0}$ which
is not possible. Thus, $H$  has common edge with $C$ and
therefore $H\subseteq C$.

Note that  we obtain ``every cycle which contains $e$, also contains $H$". By considering the part
\ref{part:cycle segments and path segments are seperable},  it follows that $H$  contains $e$ which is not possible. Then we get the proof of part  \ref{part:Graph mines a cycle segment is bridgeless}.
\\[.3cm]%
{\bf Proof of Part \ref{part:each connected component at most one connected component }: }
On the contrary, assume that $G_0$ is a connected component of $G-H$ and there exist two
$C_1, C_2\subseteq G_0$,  where $C_1$ and $C_2$ are two different connected components
of $C-H$. In this case, we have
\begin{align*}
C=C_1 \cup C_2 \cup H_1 \cup H_2,
\end{align*}
where $H_1$ and $H_2$ are connected components of $H-(C_1\cup C_2)$. Define
\begin{align*}
\overline{C}:= (C_1\cup C_2)\cup (H_1\cup P),
\end{align*}
where $P$ is a path body in $G_0$ such that connects endpoints of $C_1$ and
$C_2$ which are not endpoints of $H_1$. By the part \ref{part:Graph mines a cycle segment is bridgeless},
$G_0$ is bridgeless and the  Menger theorem (Max-flow Min-cut theorem) guarantees  the existence of such $P$. On the other hand, one can show that $\overline{C} \in \cycle(G)$ holds. But this contradicts with the fact that $H$ is cycle segment.
\\[.3cm]%
{\bf Proof of Part \ref{part:same number of connected components}:} Let $\mathfrak{N}(H):=s$. In the case of  $s = 1$, by part
\ref{part:Graph mines a cycle segment is bridgeless} and  Menger's theorem, we get the proof.

Assume that $s > 1$, i.e., $H$ is not connected. Let us define
\begin{align*}
\bigcup_{i = 1}^s H_i:=H,\qquad \bigcup_{i = 1}^r G_i:=G - H,
\end{align*}
where $\{H_i\}_{i = 1}^s$ and  $\{G_i\}_{i = 1}^r$ are connected components of
$H$ and $G-H$, respectively. We are going to show that each $G_i$ has at least two common vertices with connected components of $H$. Otherwise,   we have three main cases as follows:
\begin{enumerate}
\item Let for some $1\leq i_0\leq r$,  $G_{i\ts{0}}$ have no common vertex with any of $\{H_i\}_{i = 1}^s$. Then $G$ is not connected which is impossible;
\item Let for some $1\leq i_0\leq r$,  $G_{i\ts{0}}$  have one common vertex with one $H_{j\ts{0}}$, $(1\leq j_0\leq s)$. Then
each edge of $H_{j\ts{0}}$  is a bridge for $G$ which is impossible.
\item Let for some $1\leq i_0\leq r$,  $G_{i\ts{0}}$ have one common vertex with $H_{j\ts{0}}$  and $H_{j\ts{1}}$ $(1\leq j_0<j_1\leq s)$
(i.e., the vertex is common in $G_{i\ts{0}}$, $H_{j\ts{0}}$ and $H_{j\ts{1}}$).  In this
case, since $G$ is a strong cyclic graph then there exists $C\in \cycle(G)$ such that
contains exactly one of $H_{j\ts{0}}$ or $H_{j\ts{1}}$. Also, $C$  has common
edges with $G_{i\ts{0}}$ (because otherwise $G_{i\ts{0}}$  becomes a
cycle component of $G$ which is impossible). Since $H_{j\ts{0}}$ and $H_{j\ts{1}}$
lie in $H$ and $H$ is a cycle segment, then we get a contradiction.
\end{enumerate}
According to these three main cases, it follows that  $G_i$ has at least two common vertices with connected components of $H$. \\\\

Every common vertex  between any $G_i$ $(1\leq i\leq r)$ and $H_j$ $(1\leq j\leq r)$ is an endpoint of $H_j$. Otherwise,  such a vertex must be cycle generic and it is in the middle of a path segment- i.e., $H_j$- which  is impossible. Now, let us define
\[
\mathcal{B}:=\Big\{v\in V|\ v \text{ is common in some $G_i\ (1\leq i\leq r)$ and $H_j\ (1\leq j\leq s)$}\Big\}.
\]
Let us enumerate the members of $ \mathcal{B}$ throughout by two different methods. In the first method, for every $G_i\ (1\leq i\leq r)$, we have at least two members of $\mathcal{B}$ which yields $2r\leq \# \mathcal{B}$. In the second one, for every $H_j\ (1\leq j\leq s)$, we have exactly two members of $\mathcal{B}$ which yields $\# \mathcal{B}=2s$. Thus, $r\leq s$. On the other hand, the part \ref{part:each connected component at most one connected component } implies that  $s \leq r$. Then  $s = r$ and we get (\ref{P5}). By considering the part \ref{part:each connected component at most one connected component }, we get the rest of proof.
\\[.3cm]%
{\bf Proof of Part \ref{part:at most one not connected cycle segment}:} If all of the cycle segments of
$G$ are connected, then nothing remains to prove. On the contrary, let us assume that
$H$ is a disconnected cycle segment of $G$. 
We are going to prove that all of the other path segments lie  in different cycle segments, which implies that they are cycle segments. This implies that there is not exist any disconnected cycle segment other than $H$, and then we will get the proof.

In order to prove the above claim, one  notes that every $G_i$ is bridgeless, where $G - H=\bigcup_{i = 1}^r G_i$ and    $\{G_i\}_{i = 1}^r$ are the connected components of $G-H$. Let $P_1$ and $P_2$ be two different path segments of $G$ (out of $H$). Then, we have two cases as follows:\\\\
{\bf Case 1:}  Let $P_1$ and $P_2$ belong to  two different $G_i$. Then one can find that $C_1, C_2\in \cycle (G)$ $(C_1 \neq C_2)$ which contains $P_1$ and $P_2$, respectively. This completes the proof.\\\\
{\bf Case 2:} Let $P_1$ and $P_2$ belong to the same connected component  $G_{i_0}$'s. Suppose that $H\subseteq C\in \cycle(G)$ and $P = C\cap G_{i_0}$ ($P$ is not empty by the part \ref{part:same number of connected components}). Then, one  can replace
$P$ with $P_1\subseteq P'_1$ and $P_2\subseteq P'_2$ such that $(C-P)\cup P'_1$ and
$(C-P)\cup P'_2$ be cycle bodies containing  $P_1$ and $P_2$, respectively.  Since $G_{i_0}$ is bridgeless, then  $P'_1$ and $P'_2$ can be chosen without any common edge with $P_2$ and $P_1$, respectively. In this case, we get the proof.
\\[.3cm]%
{\bf Proof of Part \ref{part:excistanse of removable path segment}:}
Let $G$ be a strong cyclic graph. Since  $\cdim (G)\ge 2$,  then $G$ has at least two cycle segments $H_1$ and $H_2$. By the part \ref{part:at most one not connected cycle segment}, at least one of $H_1$ or $H_2$ is connected. Without any loss  of generality, suppose that $H_1$ is connected. It is easy to see that,  $G-H$ is a strong cyclic graph. The rest of the claim is obvious.
\\[.3cm]%
{\bf Proof of Part \ref{part:existanse of subgraph per dimension}:} We are going to use the induction on $\cycle (G)$. If $\cycle (G)=n+1$, then for $m:=n+1$ we get $G_m=G$. Let $m=n$. Then,  by the previous part,  $G_m=G-P$. By applying the  induction hypothesis on $G_n$, we get the proof.
\\[.3cm]%
{\bf Proof of Part \ref{part:separating cycle-path segment}:} In the case of $\cdim (G)=3$, the proof is trivial. Let  $\cdim (G)>3$. Then, we have two cases as follows:\\
{\bf Case (1):} Let $G$ be a strong cyclic graph. In this case, it is crystal clear that
$C$ does not contain all path segments. Then, there exists  a connected cycle segment out of $C$, namely $P$. Thus $G-P$
is what we want.\\
{\bf Case (2):} Suppose that $G$ is not a strong cyclic graph. By using the same method used in case (1) on any cycle component of $G$, we get the proof.\\\\
\smallskip
\noindent
{\bf Proof of Part \ref{part:path segment removal involve at most 4}:} One can see that if the endpoints of $w$ remain cycle generic in $G-H$, then $w$  is a path segment for $G-H$. Let us remark that  $G$ is a connected and  bridgeless graph with $\cdim (G)\geq 2$. If $w$ is not a path segment of $G-H$, then for one of its endpoints, namely  $v$,  one needs to have $\deg_{G-H}(v) = 2$. On the contrary, assume that there are cycle segments $\{H\}^5_{i=1}$ such that $w$ is not a path segment for $G-H_i$, $(1\leq i\leq 5)$. Suppose that $v_0$ and $v_1$ are the endpoints of $w$. Let us remark that the cycle segments are edge disjoints. If $\deg_G(v_0)=\deg_G(v_1)=3$, then we have at most 4 cycle segments, namely $H$, such that $w$ is not a path segment of $G-H$. Then,  at least one of  $\deg_G(v_0) \geq 4$ or $\deg_G(v_1) \geq 4$  holds. Without loss of generality, suppose that $\deg_G(v_0) \geq 4$   holds. By the part 5, at last one of that 5 cycle segments is not a path segment. Then  $v_0$ remains cycle generic in at least 4 cases (out of 5). This means that $v_1$ is not a cycle generic vertex of $G-H_i$ for those 4 cases. It follows that $v_1$ satisfies $\deg(v_1) \geq 5$. This is a contradiction.
\\[.3cm]%
{\bf Proof of Part \ref{part:path segment-intersection of two cycle}:}
In order to prove this part, it is sufficient to  prove it for a strong cyclic graph.
If $H$ is not a connected cycle segment of $G$, then by the similar method used in the part 4, we get the proof.
If $H$ is connected cycle segment of $G$, then $H$ is also a path segment. If  $G$ has less than 7 cycle segments, then it is easy to see that the proof holds. Now, assume that $G$ has at least 7 cycle segments. Thus, $G$ has at least 6 path segments which are also cycle segments. By the part 9, there exists path segment $P$ (which is also a cycle segment) in $G$ such that $H$ is a path segment in $\widehat{G}:=G-P$. By using the induction on the number of cycle segments of $\widehat{G}$, we get the proof.
\\[.3cm]%
{\bf Proof of Part \ref{part:special generator}:} For the case of $\cdim (G)=2$, proof is trivial. We are going to use induction on $\cdim (G)$. Assume that
our claim holds  for $\cdim (G)\le m$. Let $\cdim (G)=m+1$. Suppose that $P$ is a path segment in $C$. Then
by part \ref{part:path segment-intersection of two cycle},  there exists a simple
cycle $C_0$ in $G$ such that $P = C\cap C_0$. Since $G$ is a cactus-free graph, then $C_0\neq C$. Let us define
$\widehat{G}:=G-P$ and $\widehat{C}:=(C\cup C_0)-P$. Obviously, $\widehat{G}$ remains a connected, bridgeless and cactus-free  graph. By part (viii) of Proposition \ref{prop:first and basic},  for any  optimal sign labeling  $f$ of $G$,  ${\hat f}:=f|_{\cycle(\widehat{G})\times E_{\widehat{G}}}$  is an optimal
sign labeling  of $\widehat{G}$.  By applying the induction hypothesis on $\widehat{G}$, one can obtain the existence of
$\{\widehat{C}'_j\}_{j=1}^m$ and $\{\widehat{C}_j\}_{j=1}^m$
such that 11(a), 11(b) and 11(c) hold.
 Now, let us define
\begin{equation*}
C_i:=\left\{%
\begin{array}{lcl}
\widehat{C}_{i-1}& & i > 2\\
C_0& & i = 2\\
C& & i = 1
\end{array}
\right.\quad,\quad%
C_i':=\left\{%
\begin{array}{lcl}
\widehat{C}_{i-1}'& & i > 2\\
\widehat{C}& & i = 2\\
C& & i = 1
\end{array}
\right.\qquad(1\le i\le m+1).%
\end{equation*}
It is easy to see that,  $\{{C}'_j\}_{j=1}^{m+1}$ and $\{{C}_j\}_{j=1}^{m+1}$ satisfy 11(a), 11(b) and 11(c).
\end{proof}

\bigskip

Here, we present a characterization formula that describes $\cdim (G)$.
\begin{cor}\label{cor:calculating Cdim w.r.t deg}
Let $G$ be a graph. Then the following holds
\begin{align}
\sum_{v\in V^{\!c}}\big(\degc v - 2\big)=2\big(\cdim (G) -1\big).\label{EC}
\end{align}
\end{cor}
\begin{proof}
The proof is trivial for the case of $\cdim (G)\leq 1$. Now, if (\ref{EC}) holds for each connected component of a
graph, then it holds for entire graph. By considering the part \ref{part:excistanse of removable path segment}
of  Proposition \ref{prop:important} and the part (ix) of Proposition \ref{prop:first and basic}, we get the proof.
\end{proof}


\smallskip
\begin{cor}
Let $G$ be a strong cyclic graph and  $H$ be an arbitrary  cycle segment of $G$. Then, any connected component of $H$  is either a single vertex or
exactly one path segment of $G$. In particular, if two path segments of $G$ lie in $H$, then they lie in different connected components of $H$.
\end{cor}
\begin{proof}
Let $w_1$ and $w_2$ be two different path segments, i.e.,  $G(w_1) \neq G(w_2)$, such that $G(w_1), G(w_2)\subseteq H$. We are going to show  that
they have not any common endpoint. On the contrary, assume that $w_1$ and $w_2$ have at least one common endpoint. Without loss of generality, one can assume that $w:=w_1 + w_2$ is well defined and $v$ is the  vertex connecting  $w_1$ and $w_2$. Since  $v$ is cycle generic, then  there exists $e\in E$ outside of $w_1$ and $w_2$ such that $v\in r(e)$. Let us consider $C\in \cycle(G)$
such that contains $e$ and $w_1$. It is easy to see that $G(w_2)\nsubseteq C$. This means that $w_1$ and $w_2$ are not in the same cycle segment which is impossible.

Now, if two arbitrary path segments in $H$ have not any common endpoints, then  every
connected component of $H$ may contain at last one path segment. This completes the proof.
\end{proof}

\bigskip


Now, we are ready to prove our main result.

\bigskip
\noindent
{\bf Proof of Theorem \ref{MainTHM}:} For this aim, it is sufficient to prove it for a strong cyclic graph.  Let $G$ be a
strong cyclic graph. If $\cdim (G)\le 2$, then we get the proof.

For the rest of proof, we use the induction on cyclic dimension. Now, assume that the cycle double cover conjecture holds for $G$ in the case of
$\cdim (G)\le m$. We are going to  prove the cycle double cover conjecture for the case that $\cdim (G) = m + 1$ holds.
By applying the part
\ref{part:special generator} of Proposition \ref{prop:important} for an arbitrary cycle $C$ in $G$, we get  $\{C_i\}_{i=1}^{m+1}$
and $\{C_i'\}_{i=1}^{m+1}$. Then, it is easy to see that $\{C_i\}_{i=1}^{m+1}\cup\{C_{m+1}'\}$
is a cycle double cover for $G$. This completes the proof.
\qed

\bigskip

\noindent
Abdollah Alipour and Akbar Tayebi \\
Department of Mathematics, Faculty  of Science\\
University of Qom \\
Qom. Iran\\
Email:\ \ alipour.ab@gmail.com\\
Email: akbar.tayebi@gmail.com
\end{document}